\documentclass[11pt]{amsart}
\usepackage{graphicx} 
\usepackage{mathtools,amssymb,amsthm}

\title[Global MT for Indicator Weights of Bounded Complexity in $\mathbb{R}^2$]{Global Mizohata-Takeuchi for Indicator Weights of Bounded Complexity in $\mathbb{R}^2$}
\author{Jiaqi Hao, Ruyi Wang and Ruixiang Zhang}
\date{\today}

\newcommand{\R}{\mathbb{R}}

\newcommand{\mi}{2\pi\mathrm{i}}

\theoremstyle{definition}

\newtheorem{theorem}{Theorem}[section]
\newtheorem{definition}[theorem]{Definition}
\newtheorem{lemma}[theorem]{Lemma}
\newtheorem{conjecture}[theorem]{Conjecture}

\newtheorem{proposition}[theorem]{Proposition}
\newtheorem{remark}[theorem]{Remark}
\usepackage[
    backend=biber,
    style=alphabetic
]{biblatex}
\begin{document}

\maketitle

\begin{abstract}
    We prove the global Mizohata-Takeuchi for bounded complexity weights in $\mathbb{R}^2$, with uniform implied constant in terms of the complexity and an arbitrary $C^2$ uniformly convex curve on the Fourier side. A $\log$-loss version follows from a short argument by the $T^*T$ method and the main estimate of \cite{BGZZ2026}. The main effort of this paper is devoted to the removal of the logarithmic factor. To achieve this, we use tools in semialgebraic geometry and adopt and develop ideas in \cite{RicciStein1987}. As a by-product, we obtained a uniform geometric inequality for planar regions of bounded complexity that may be of independent interest.
\end{abstract}

\section{Introduction}
\label{intro}
\subsection{Main result}
Let $\Sigma\subset\mathbb R^n$ be a compact $C^2$ hypersurface with  hypersurface measure $d\sigma$. Define its Fourier extension operator
\[
    Ef(x) 
    :=\int_\Sigma e^{\mi x\cdot\xi}f(\xi)\,d\sigma(\xi),
    \qquad x\in\mathbb R^n.
\]

A natural question in the $L^2$ theory is to prove weighted $L^2$ inequalities to understand where the mass of $E f$ can concentrate. This question makes sense even if $\Sigma$ is a hyperplane. In this case, a correct and most natural weighted $L^2$ estimate is

\begin{equation}
    \label{MTeq}
    \int_{\mathbb R^n}|E f(x)|^2w(x)\,dx
    \lesssim_\Sigma
    \|Xw\|_{L^\infty}\|f\|_{L^2(\Sigma,d\sigma)}^2,
\end{equation}
where $w\geq 0$ is a continuous weight on $\mathbb R^n$ and the \emph{X-ray transform} for a line $\ell$ is defined as
\[
    Xw(\ell):=\int_\ell w\,d\mathcal H^1,
    \qquad
    \|Xw\|_{L^\infty}:=\sup_{\ell\subset\mathbb R^n\;\mathrm{line}}Xw(\ell).
\]

The (global, or endpoint) \emph{Mizohata-Takeuchi conjecture} is an attempt to directly generalize this to other hypersurfaces.

\begin{conjecture}[Mizohata-Takeuchi Conjecture]\label{MTConj}
    \eqref{MTeq} is true for every compact, uniformly convex $C^2$ hypersurface. 
\end{conjecture}

If true, Conjecture \ref{MTConj} will be useful in many applications in areas like analysis, PDE and number theory. Surprisingly, Cairo \cite{Cairo2025} disproved Conjecture \ref{MTConj} for any connected hypersurface patch, unless it is a subset of a hyperplane. Nevertheless, it is of value to find useful classes of weights where this conjecture holds. In the present paper, we prove this conjecture in $\mathbb{R}^2$ for weights being characteristic functions of \emph{semialgebraic sets of bounded complexity}.

We introduce one notational simplification before stating our main theorem. For an open set $U\subset\mathbb R^2$, we define its maximal line capacity
\[
    L_U:=\sup_{\ell\subset\mathbb R^2\;\mathrm{line}}|\ell\cap U|=\|X\chi_U\|_{L^\infty}.
\]

\begin{theorem}[Main Theorem]\label{endpointMTthm}
    Let $\gamma\in C^2([-1,1])$ satisfy
$1\leq\gamma''(\xi)\leq2, \xi\in[-1,1]$. Its graph $\Sigma=\{(\xi,\gamma(\xi)):\xi\in[-1,1]\}$  is a $C^2$ uniformly convex curve. Define the extension operator
\[
    E_\Sigma f(x,y):=\int_{-1}^1 e^{\mi(x\xi+y\gamma(\xi))}f(\xi)\,d\xi.
\]

For a fixed constant $k$, let $U\subset\mathbb R^2$ be an open
    semialgebraic set of complexity at most $k$. Then, for every
    $f\in L^2([-1,1])$,
\begin{equation}\label{EndpointMTeq}
    \int_U|E_\Sigma f(x,y)|^2\,dx\,dy
        \lesssim_k L_U\|f\|_{L^2([-1,1])}^2.
\end{equation}
\end{theorem}

We remark that if we assume $\mathrm{supp}\,w \subset B_R$ (a ball of radius $R>1$) and allow a $\log R$-loss in \eqref{EndpointMTeq}, then there is a much shorter proof for \eqref{endpointMTthm}, see Subsection \ref{pfoutlinesubsec}. Most of the effort in this paper is devoted to proving the endpoint inequality without $\log$-loss.

\subsection{Related historical background} Conjecture \ref{MTConj} originated in the study of
dispersive equations and their Cauchy problems \cite{Takeuchi1974, Takeuchi1980, Mizohata1985}; see also \cite{BarceloRuizVega1997} for its connection with weighted estimates for the Helmholtz equation. For the sphere, this conjecture known for radial weights \cite{BarceloRuizVega1997, CarberySoriaVargas2007}.

In the talk \cite{guth-barrier-22}, Guth surprisingly pointed out that  axiomatic wave packet decomposition alone cannot prove Conjecture \ref{MTConj}. The outline in \cite{guth-barrier-22} was then fleshed out and further developed in \cite{CarberyIliopoulouWang2024}, where the authors proved: if $\Sigma$ is a compact, uniformly convex $C^2$ hypersurface, then
\begin{equation}\label{powerlossineq}
    \int_{B_R}|E_\Sigma f(x)|^2w(x)\,dx
    \lesssim_{\Sigma,\varepsilon}
    R^{\frac{n-1}{n+1}+\varepsilon}
    \|Xw\|_{L^\infty}\|f\|_{L^2(\Sigma,d\sigma)}^2
\end{equation}
and the exponent $\frac{n-1}{n+1}$ cannot be improved by the usual axiomatic wave packet decomposition. In the setting of the hyperbolic paraboloid in $\mathbb R^3$, Kim-Wang \cite{KimWang2026} obtained a bound with power loss $R^{\frac12+\varepsilon}$ for Conjecture \ref{MTConj}.

Cairo \cite{Cairo2025} then constructed weights that fail Conjecture \ref{MTConj} for all non-hyperplanar $\Sigma$ in all dimensions. Her counterexamples have $\log$-blowups. Later Cairo-Zhang \cite{CairoZhang2025} showed that there are examples of strictly convex $C^2$ hypersurfaces where \eqref{powerlossineq} fails with the exponent $\frac{n-1}{n+1}$ replaced by anything smaller. See \cite{Fenves2026, bennett2026rectangles, fenves2026sharp} for related work on counterexamples to Conjecture \ref{MTConj}. See references in the papers mentioned in this section for more related history of Conjecture \ref{MTConj}.

\subsection{Proof outline}\label{pfoutlinesubsec}


We outline the ideas in the proof of Theorem \ref{endpointMTthm}. By considering intersection with a large ball, we may assume $U$ is bounded without loss of generality. We use the $TT^*$ method. Let $E_Uf=(E_\Sigma f)|_U$. Define $T:=E_U^*E_U.$ Then
\[
    Tf(x)=\int_{-1}^1 K(x,y)f(y)\,dy,
\]
where
\begin{equation}\label{formulaofK}
    K(x,y):=\int_U e^{\mi\left(\xi(x-y)+\eta(\gamma(x)-\gamma(y))\right)}\,d\xi\,d\eta, \qquad x,y \in [-1,1].
\end{equation}

It suffices to prove
\begin{equation}\label{TTstarineq}
    \|T\|_{L^2([-1,1])\to L^2([-1,1])}\lesssim_k L_U.
\end{equation}

If we assume $U \subset B_R (R>1)$ and are allowed to prove \eqref{TTstarineq} with a $\log R$-loss, there is a quite simple proof. By the stationary set method (Theorem 1.1 in \cite{BGZZ2026}) and the trivial bound $|U|\lesssim R L_U$, we obtain \[|K(x, y)|\lesssim \min \{RL_U, \frac{L_U}{|x-y|}\}\] and the Schur test gives
\begin{equation}\label{loglossversion}
    \|T\|_{L^2([-1,1])\to L^2([-1,1])}\lesssim_k L_U\cdot \log R.
\end{equation} The above non-endpoint proof even works without the strictly convex condition $\gamma '' > 1$.

\begin{remark}
    We remark that the above application of the $TT^*$ method is different from some classical applications of that method. In some other applications to Fourier restriction problems, we could have considered another useful operator $E_U E_U^*$ and would have a convolution operator with the kernel being the more familiar $(d\sigma)^{\hat{}}$. But its oscillation  is very difficult to deal with in this problem. Instead, in the current consideration of $E_U^* E_U$, the kernel $K$ may seem less elementary at first, but Theorem 1.1 in \cite{BGZZ2026} guarantees that it has much better decay property. It is crucial to take advantage of this property in both the non-endpoint and the endpoint proofs in the present paper. This method was partly inspired by the recent work \cite{guth2026new}, which also fits into the broader picture of  weighted restriction. 
\end{remark}

To obtain a uniform endpoint estimate \eqref{TTstarineq} is much more difficult. We need to find more delicate cancellations in the operator $T$ and take advantage of them. For example, it is natural to expect $K$ sometimes (but not always) behave like the Hilbert transform. Uniformity is also a difficulty: In the estimate we obtained, the constant only depends on the complexity of $U$. Much of the effort in this paper will be devoted to prove the endpoint estimate.

We outline some ideas in the proof of the endpoint estimate. Continue with the idea in \cite{BGZZ2026}, we can express $K$ as a one-dimensional integral and see if integration-by-parts can help. When $|x-y|$ is small, we do not hope integration-by-parts to be a wise move. Quantifying this observation, we were suggested to prove the following curious estimate

\begin{equation}\label{AreaboundedbyL2}
    \mathrm{Area}(U) \lesssim_k L_U^2
\end{equation}
that may be of independent interest.

A familiar sharp example for \eqref{AreaboundedbyL2} is a square. Perhaps it is more surprising that \eqref{AreaboundedbyL2} also saturated by the example $U$ being an annulus (with arbitrary inner and outer radii). We were able to prove \eqref{AreaboundedbyL2} and use it to bound the ``non-oscillatory'' contribution of $K$ (i.e. when the integrand in \eqref{formulaofK} is not expected to oscillate much). Because of this special property of the annuli, in our treatment of the oscillatory contribution, much of our proof intuition was learned from explicit calculations related to the annulus.

The proof of \eqref{AreaboundedbyL2} is done in Subsection \ref{neardiagsubsec}. Intuitively, in the proof we see that a bounded complexity open set on the plane either looks like a rectangle or looks like the annulus. When its boundary is flat, it looks like a rectangle and we have a good intuition for the largest X-ray transform. When its boundary is highly curved, it looks like the annulus. A calculation shows that if the region does not hold long segments, its boundary must be seriously curved. Since the boundary of $U$ has bounded complexity, the total variation of the tangent direction is well controlled. This ends up being sufficient to prove \eqref{AreaboundedbyL2} in the very-curved-boundary case.

For the oscillatory contribution (handled in Section \ref{oscillationsec}), it is convenient to use a layer-cake type decomposition to further decompose $K$ and reduce to suboperators. This is performed in Subsection \ref{offdiagsubsec}. These suboperators have kernels look like the Hilbert transform, but with a modulation jointly in $x$ and $y$ on the numerator. We handle this contribution use a method closely related to that of Ricci-Stein \cite{RicciStein1987}.   Notice that if the modulation were in the variable-separated form of $F(x)-G(y)$, the problem would become boundedness of Hilbert transform after modulation and would therefore be solved.  We end up using $F(x)-G(y)$ to approximate the modulation when $x$ and $y$ are close, and use $TT^*$ again to control contributions from faraway $x$ and $y$.

There are many more technical details in estimating the oscillatory contribution, often caused by our requirement that the estimate needs to be $\log$-free and uniform in terms of the complexity $k$. We need to do Whitney-type decompositions and further technical approximations. There is a uniform result in tame geometry that we used, recorded in Section \ref{complex}: For a semialgebraic function $f$ of complexity $O(1)$ in one variable $x$ on an interval $I$ and a fixed derivative order $s$, we can partition $I$ into $O(1)$ disjoint intervals, such that inside each interval $J$, 

\begin{equation}\label{derivatebounduniform}
    |\frac{f^{(s+1)}(x)}{f^{(s)}(x)}|\lesssim \frac{1}{\text{dist} (x, \partial J)}.
\end{equation}
This estimate is proved with the help of complex analysis and may be of independent interest.

\subsection{Organization of the paper}

The paper is organized as follows. Section \ref{intro} presents our main theorem (Theorem \ref{endpointMTthm}) and reviews its background. Section \ref{complexity} collects the results from semialgebraic geometry used in the paper. Section \ref{example} presents key examples that motivate our proof along with some counterexamples for more general estimates and cautions. Section \ref{reduction} proves the geometric estimate \eqref{AreaboundedbyL2} and use it to establish the non-oscillatory estimate Theorem \ref{area} and reduces the oscillatory case to the singular integral estimate stated in Theorem \ref{main}.  Section \ref{oscillationsec} proves Theorem \ref{two endpoints}, thereby reducing the proof of Theorem \ref{main} to Theorem \ref{curvature-partition} concerning \eqref{derivatebounduniform}, which is proved in the  Section \ref{complex}.

\subsection{Communication with AI}
The starting point of this work was before AI communications. We were interested in finding new classes of weights for which Conjecture \ref{MTConj} holds, and observed the main theorem in  \cite{BGZZ2026} with the $TT^*$ method give the $\log$-loss version of \eqref{loglossversion}. We then tried to remove the $\log$ factor. This turned out to be a difficult problem. So we took a step back to work on the special case where the curve is the unit parabola and the weight is the characteristic function of the square, the disc or the annulus (with arbitrary inner and outer radii). The global versions  without logarithmic losses in these cases were all successfully proved, with two very different methods. The square has corners and the  $TT^*$ operator is related to the Hilbert transform, which is known to be bounded on $L^2$. In the disc or annulus case, integration by parts was used to control the kernel. We also wondered about the general case. Even the non-oscillatory part seemed to suggest the nice inequality \eqref{AreaboundedbyL2}, which we surprisingly found difficult and absent from the existing literature. 

Claude Fable 5 was used to give the proof of \eqref{AreaboundedbyL2} and Theorem \ref{area}, which handles the non-oscillatory part. For the oscillatory part, inspired by the case study of the annulus, we 
were naturally led to think Lemma \ref{annulus} is true,  but we were not able to prove it. We then tried other arguments that worked for an annulus and a few other special cases, but could not make them work in general. Claude Fable 5 examined these arguments and suggested using ideas from \cite{RicciStein1987} to approach the problem through Lemma \ref{annulus}, but could not complete the proof either.

A while later, we used ChatGPT 6 Astra to continue this approach. It gave a proof of Theorem \ref{main} for the extension operator associated with the parabola, using the third-order approximation explained in Subsection \ref{thirdordersubsec} and  partition Theorems \ref{algebraic-log-derivative} and \ref{curvature-partition} in Section \ref{complex}. Some ideas we had earlier, inspired by the case studies, turned out to appear in the proof of the oscillatory part, such as integration by parts, approximating the kernel by Hilbert transform with modulation and relating the line intersection length in nearby different directions and positions. We worked through the proof and extended it to extension operators associated with curves \(\gamma\) satisfying \(1\leq\gamma''\leq2\). In a separate exchange after giving the proof, ChatGPT 6 Astra also provided references, \cite{Pan1991} and \cite{Pan1995} where Lemma \ref{annulus} was proved. We also used ChatGPT to improve the language and presentation of the paper. 

We have checked all proofs, including those obtained with AI assistance, and verified the references. We take full responsibility for the correctness of the results and the content of the paper.

\subsection{Acknowledgments}
JH is deeply grateful to Zhongkai Tao and Zixuan Wang for their continued support and many insightful and helpful discussions. RW thanks Shixin Ren and Tianchi Xiang for their technical support and stimulating discussions on the use of AI. She also thanks Alex Cohen and Hong Wang for their helpful comments on the manuscript.
RZ is supported by NSF CAREER DMS–2143989 and Sloan Research Fellowship.

\bigskip
\section{Background in semialgebraic geometry}
\label{complexity}
In this section, we recall several fundamental tame geometric properties of semialgebraic sets and functions in $\mathbb{R}^2$.

\begin{definition}[{Semialgebraic Sets and Complexity in $\mathbb{R}^2$, \cite[Definition 2.1]{BGZZ2026}}]
A \emph{semialgebraic set} $U \subset \mathbb{R}^2$ is a finite union of subsets, each of which is defined by a formula of the form 
\[ P_1(x, y) = 0, \dots  , \, P_n(x,y)=0, \quad Q_1(x, y) > 0, \dots, \, Q_l(x, y) > 0 \]
where $P_1, \dots, P_n,  Q_1, \dots, Q_l \in \mathbb{R}[x, y]$. 

If the total number of polynomials occurring in the formulas defining $U$ is bounded by $s$, and the maximum degree of these polynomials is bounded by $D$, we say that the \emph{complexity} of $U$ is bounded by $k = sD$.
\end{definition}

\begin{definition}[{Semialgebraic Functions, \cite[Definition 2.2]{BGZZ2026}}]
A semialgebraic function $f: I \to \mathbb{R}$ (where $I \subset \mathbb{R}$ is an interval) is a function whose graph 
\[ \text{graph}(f) = \{(x, f(x)) \in \mathbb{R}^2 : x \in I\} \]
is a semialgebraic set in $\mathbb{R}^2$. We say that the complexity of $f$ is bounded by $k$ if the complexity of the semialgebraic set $\text{graph}(f)$ is bounded by $k$.
\end{definition}

\begin{definition}
Let \(I\subset\R\) be an interval. A function \(f:I\to\R\) is called an algebraic function if there exists a nonzero polynomial \(P\in\R[x,y]\) such that
\[
P(\xi,f(\xi))=0,\; \forall \xi\in I.
\]
\end{definition}

\begin{theorem}[{\cite[Propositions 2.83 and 7.31]{BasuPollackRoy2006}}]
\label{component}
Let \(I\subset\R\) be an interval and let \(f:I\to\R\) be a semialgebraic function of complexity at most \(k\). Then, for every \(t\in\R\), the set
\[
\{x\in I:f(x)>t\}
\]
is semialgebraic, and the number of its connected components is bounded by \(O(k^2)\).
\end{theorem}

\begin{theorem}[{\cite[(2.1.4)~and ~(2.4.1)]{FernandoGamboaRuiz2014}}]
\label{algebraic}
Let \(f:I\to\R\) be a semialgebraic function, where \(I\subset\R\) is an interval. Then there exists a finite partition
\[
I=I_1\cup\cdots\cup I_N\cup F,
\]
where each \(I_j\) is an interval and \(F\) is a finite set, such that \(f|_{I_j}\) is real analytic and algebraic for every \(j\). Moreover, \(N\) and the degrees of the corresponding polynomial equations can be bounded in terms of the complexity of \(f\).
\end{theorem}

Next we will give the theorems of cell decomposition, which ensure that any semialgebraic set of bounded complexity can be partitioned into finitely many topologically simple pieces, called \emph{cells}. 

\begin{theorem}[{Cell Decomposition in $\R$, \cite[Corollary 2.79]{BasuPollackRoy2006}}]
   $E\subset \R$ is semialgebraic if and only if it can be decomposed into a finite union of points and intervals. 
\end{theorem}
\begin{definition}[Cells in $\mathbb{R}^2$]
A subset $C \subset \mathbb{R}^2$ is a \(C^r\) \emph{cell} if it takes one of the following forms:
\begin{enumerate}
    \item (0,0)-cell: $C = \{(x_0, y_0)\}$.
    \item (0,1)-cell: $C = \{x_0\} \times I$, where $I \subset \mathbb{R}$ is an open interval.
    \item (1,0)-cell: $C = \Gamma(f)$, the graph of $f: I \to \mathbb{R}$, a \(C^r\) semi-algebraic function defined on an interval.

    \item (1,1)-cell: $C = \{(x, y) \in I \times \mathbb{R} : f(x) < y < g(x)\}$, where $f, g: I \to \mathbb{R}\cup \{\pm\infty\}$ are \(C^r\) semi-algebraic functions satisfying $f(x) < g(x)$ for all $x \in I$.
\end{enumerate}
\end{definition}

\begin{theorem}[{Quantitative $C^r$ Cell Decomposition in $\mathbb{R}^2$, \cite[Theorem 5.34, Remark 5.35, Algorithm 11.15 and Exercise 11.8]{BasuPollackRoy2006}}]
\label{cell}
(1) Let $U \subset \mathbb{R}^2$ be a semi-algebraic set of complexity at most $k$. Then there exists a partition of $\mathbb{R}^2$ into at most $N_1(k,r)$ \(C^r\) cells such that $U$ is a union of some of these cells. Furthermore, the complexity of each cell is uniformly bounded by a constant depending only on $k$.

(2) Let $f: \mathbb{R} \to \mathbb{R}$ be a semi-algebraic function of complexity at most $k$. Then there exists a partition of $\mathbb{R}$ into at most $N_2(k, r)$ intervals and points such that for each open interval $I$ in the decomposition, the restriction $f|_I \in C^r(I)$.

\end{theorem}
\begin{theorem}[A consequence of {\cite[Corollary 5.7]{BasuPollackRoy2006}}]
\label{slice}
Let \(S \subset \mathbb{R}^n \times \mathbb{R}\) be a bounded
semialgebraic set. For each \(x \in \mathbb{R}^n\), define the fiber
\[
S_x := \{t \in \mathbb{R} : (x,t) \in S\}.
\]
Then the function
\[
h(x) := |S_x|
\]
is semialgebraic, where \(|\cdot|\) denotes the one-dimensional Lebesgue
measure.
\end{theorem}

\bigskip
\section{Some examples and cautions}
\label{example}
In this section, we will give a brief discussion of some important examples, such as the square and annulus. Through calculating these examples we can get some motivation for the general proof. To make the calculation more clear, we only consider the parabola case, where \(\gamma(x)=x^2\).
\subsection{Square}
We begin with the square
\[
U=[-R,R]^2\subset\mathbb R^2,\quad R\ge1.
\]
Since \(U\) is a product, separation of variables gives
\[
K(x,y)
=\frac{1}{\pi^2}
\frac{\sin(2\pi R(x-y))}{x-y}
\frac{\sin(2\pi R(x^2-y^2))}{x^2-y^2}.
\]

For sufficiently large \(R\), taking \(x=0\) and
\(R^{-1}\le |y|\le(2\pi R)^{-1/2}\), we obtain
\[
|K(0,y)|
\sim R\frac{|\sin(2\pi Ry)|}{|y|}.
\]
Consequently,
\[
\int_{-1}^{1}|K(0,y)|dy
\gtrsim
R\int_{1/R}^{(2\pi R)^{-1/2}}
\frac{|\sin(2\pi Ry)|}{y}dy
\gtrsim R\log R.
\]
Thus a direct application of Schur's test causes a logarithmic loss. To obtain the desired bound of order \(R\), we need to retain the oscillatory cancellation.

On the region \(|x^2-y^2|>R^{-1}\), we have
\[
|K(x,y)|
\lesssim \frac{1}{|x-y|^2\,|x+y|}.
\]

A standard splitting of the integration region gives
\[
\sup_{x\in[-1,1]}
\int_{\Omega_x}\frac{dy}{|x-y|^2\,|x+y|}
\lesssim R
\]
where \(\Omega_x:=\{y\in[-1,1]:|x^2-y^2|>R^{-1}\}\). By Schur's test, this part of the operator has \(L^2\) norm \(O(R)\).

On the remaining region, the kernel resembles
\[
R\frac{\sin(2\pi R(x-y))}{x-y}
=\frac{R}{2\mathrm{i}}
\left(
\frac{e^{\mi R(x-y)}}{x-y}
-\frac{e^{-\mi R(x-y)}}{x-y}
\right).
\]
The two exponential terms correspond to modulated Hilbert transforms, and modulation preserves the \(L^2\) norm. The restriction \(|x^2-y^2|\le R^{-1}\) introduces an additional truncation, but it can still be handled by some standard techniques.

This example shows how we can use the modulated Hilbert transform. We first separate the integral into parts near and away from the diagonal \(y=x\). Near \(y=x\), we need cancellation between the two sides of \(x\), as in the Hilbert transform. Away from \(y=x\), we can bound that part without pairing these points to use their cancellation. We will use the same idea in the later proofs.
\subsection{Generalizations of the square case}
After a semialgebraic cell decomposition, we first consider cells of the form
\[
U=\{(\xi,\eta):a\le\xi\le b,\ 0\le\eta\le g(\xi)\}
\]
where \(g\ge0\) is increasing and \(C^1\). Let \(M:=\sup_{[a,b]}g\).

Define
\[
h_{x,y}(\xi):=\int_0^{g(\xi)}e^{\mi\eta(x^2-y^2)}\,d\eta.
\]
Then
\[
K(x,y)=\int_a^b e^{\mi\xi(x-y)}h_{x,y}(\xi)\,d\xi.
\]
On \(|x-y|\le M^{-1}\), the bound \(|K(x,y)|\le M(b-a)\) and Schur's test give an operator norm of \(O(b-a)\).

On the remaining region, integration by parts gives
\[
K(x,y)\mathbf1_{\{|x-y|>M^{-1}\}}
=K_a(x,y)+K_b(x,y)+\widetilde K(x,y),
\]
where
\[
\begin{aligned}
K_a(x,y)
&=-\frac{e^{\mi a(x-y)}h_{x,y}(a)}{\mi(x-y)}
\mathbf1_{\{|x-y|>M^{-1}\}},\\
K_b(x,y)
&=\frac{e^{\mi b(x-y)}h_{x,y}(b)}{\mi(x-y)}
\mathbf1_{\{|x-y|>M^{-1}\}},\\
\widetilde K(x,y)
&=-\frac{\mathbf1_{\{|x-y|>M^{-1}\}}}{\mi(x-y)}
\int_a^b
e^{\mi[\xi(x-y)+g(\xi)(x^2-y^2)]}g'(\xi)\,d\xi.
\end{aligned}
\]
Let \(T_a,T_b\), and \(\widetilde T\) be the corresponding operators. Inserting the definition of \(h_{x,y}\), each boundary term is an integral of modulated truncated Hilbert transforms. Hence
\[
\|T_a\|_{L^2\to L^2}+\|T_b\|_{L^2\to L^2}
\lesssim g(a)+g(b)\lesssim M.
\]

For the remaining term, set
\[
F_\xi(y):=e^{-\mi(\xi y+y^2g(\xi))}f(y).
\]
For each fixed \(\xi\), the inner operator becomes
\[
\int_{|x-y|>M^{-1}}
\frac{e^{\mi[\xi(x-y)+g(\xi)(x^2-y^2)]}}{\mi(x-y)}
f(y)\,dy
=
\frac{e^{\mi(\xi x+x^2g(\xi))}}{\mi}
\int_{|x-y|>M^{-1}}\frac{F_\xi(y)}{x-y}\,dy.
\]
 Minkowski's inequality gives
\[
\|\widetilde T f\|_{L^2}
\lesssim
\int_a^b |g'(\xi)|\,\|F_\xi\|_2\,d\xi
=(g(b)-g(a))\|f\|_{L^2}
\le M\|f\|_{L^2}.
\]
Therefore,
\[
\|T\|_{L^2\to L^2}
\lesssim (b-a)+M
\le L_U.
\]

This leads us to consider domains that cannot be reduced to the above case. An important example is the annulus, which we discuss next.

\subsection{Annulus}
Recall that 
\[
K(x,y)
=\int_U e^{\mi \big(\xi(x-y)+\eta(x^2-y^2)\big)}d\xi d\eta.
\]
For any \(x,y \in [-1,1]\), denote \(m:=x+y \in [-2,2]\), then we introduce the orthonormal basis
\[
e_1(m):=\frac{(1,m)}{\sqrt{1+m^2}},
\;\;
e_2(m):=\frac{(-m,1)}{\sqrt{1+m^2}}
\]
and the orthogonal change of variables $(s,t)=R_m(\xi,\eta)$ defined by
\[
(\xi,\eta)=s\,e_1(m)+t\,e_2(m).
\]
Define the slice-length function
\[h_m(s):=\big|\{t\in\mathbb R:(s,t)\in R_m(U)\}\big|,\;\;s\in\mathbb R.
\]
Then we have \(\|h_m\|_{L^\infty}\le L\) and 
\[
K(x,y)=\int_{\mathbb R} e^{\mi(x-y)\sqrt{1+(x+y)^2} s} h_m(s)ds.
\]
Notice that for the annulus, it is radial so \(\forall m\), \(h_m\) are the same function, denoted as \(h\). 

By integration by parts, we have\[
K(x,y)=-\frac{1}{\mi (x-y)\sqrt{1+(x+y)^2}}\int_{\R} e^{\mi(x-y)\sqrt{1+(x+y)^2} s} h'(s)ds
\]
and hence 
\begin{align*}
    Tf(x)&=\int K(x,y)f(y)dy\\
    &= \int \left(\int-\frac{e^{\mi(x-y)\sqrt{1+(x+y)^2} s} }{\mi (x-y)\sqrt{1+(x+y)^2}} f(y)dy \right) h'(s) ds.
\end{align*}

By Minkowski's inequality, 
\begin{align*}
    \|Tf\|_{L^2} &\le \left\|\left(\int-\frac{e^{\mi(x-y)\sqrt{1+(x+y)^2} s} }{\mi (x-y)\sqrt{1+(x+y)^2}} f(y)dy \right) h'(s)\right\|_{L_x^2L_s^1}\\
    & \le \left\|\left(\int-\frac{e^{\mi(x-y)\sqrt{1+(x+y)^2} s} }{\mi (x-y)\sqrt{1+(x+y)^2}} f(y)dy \right) h'(s)\right\|_{L_s^1L_x^2}\\
    &=\int \left\|\int-\frac{e^{\mi(x-y)\sqrt{1+(x+y)^2} s} }{\mi (x-y)\sqrt{1+(x+y)^2}} f(y)dy\right\|_{L_x^2}|h'(s)| ds.
\end{align*}

Notice that 
\[
\int |h'(s)|ds\lesssim \sup |h|\lesssim L.
\]

It suffices to show that 
\begin{lemma}[\cite{Pan1991},\cite{Pan1995}]
\label{annulus}
    For \(\forall s\), define \[
K_s(x,y)=\frac{1}{(x-y)\sqrt{1+(x+y)^2}}e^{\mi(x-y)\sqrt{1+(x+y)^2} s} 
\]
and \(T_s\) denote the operator associated with \(K_s\). Then \(T_s\)
is uniformly \(L^2\) bounded. 
\end{lemma}

Actually this is equivalent to the \(L^2\) boundedness of \(\widetilde T_s\), where 
\[
\widetilde {K}_s (x,y)=\frac{1}{x-y}e^{\mi(x-y)\sqrt{1+(x+y)^2} s} 
\]
and \(\widetilde T_s\) is defined as 
\[
\widetilde T_s f(x):=\int \widetilde K_s (x,y)f(y)dy
\]

Since we have
\[
\sqrt{1+4 x^2}\,T_sf(x)=\widetilde T_sf(x)+E_sf(x)
\]
where
\[
E_s f(x):=\int
\left(\frac{\sqrt{1+4 x^2}}{\sqrt{1+(x+y)^2}}-1\right)
\widetilde K_s (x,y)f(y)dy
\]

 Notice that
 \[
\left| \frac{\sqrt{1+4x^2}-\sqrt{1+(x+y)^2}}
{(x-y)\sqrt{1+(x+y)^2}}\right|
\le1 \]
Schur's test gives \(\|E_s\|_{L^2\to L^2}\lesssim 1\), uniformly in \(s\). Since \(\sqrt{1+4x^2}\sim 1\), multiplication by this factor is bounded and invertible, proving the equivalence.

Lemma \ref{annulus} is actually a known result. For the fixed real-analytic phase \((x-y)\sqrt{1+(x+y)^2}\), the uniform \(L^2\) estimate follows from \cite[Corollary~1]{Pan1991}; see also \cite[Proposition~2.2]{Pan1995}. This successful annulus case motivates similar reductions for more general sets. However, the implicit constant may depend on the fixed phase, and these results do not quantify this dependence in terms of complexity. In Section \ref{oscillatory}, we will prove a semialgebraic version for the class of phases arising in our reductions, with constants controlled by complexity.

Alternatively, for an annulus with inner and outer radius \(R\) and \(R+d\), a direct pointwise estimate gives
\[
|K(x,y)|\lesssim \min\left\{\frac{d\sqrt R}{|x-y|^{1/2}},\frac{\sqrt R}{|x-y|^{3/2}}\right\}.
\]
The desired bound then follows from Schur's test. For a detailed calculation and some generalizations, see the second author's bachelor's thesis \cite{wang2026weighted}.
\subsection{More general cases}
Let us now consider general semialgebraic sets \(U\). In the radial case, the slice-length function \(h_m\) is independent of \(m\), which allows us to apply Minkowski's inequality after integration by parts. For general \(U\), however, \(h_m\) depends on \(m=x+y\), so the preceding argument no longer applies directly. We therefore need a further reduction to handle this dependence. Here we just give a brief idea, and there will be a cleaner proof in Section \ref{reduction} by layer-cake and cell decomposition.

Again integration by parts gives
\[
K(x,y)
=-\frac{1}{\mi(x-y)\sqrt{1+m^2}}
\int_{\R} e^{\mi(x-y)\sqrt{1+m^2}s}h_m'(s)ds.
\]
The dependence on \(m=x+y\) prevents us from taking the weight \(h_m'(s)\) outside the integral in \(y\). On a strictly monotone branch, however, we can set \(\lambda=h_m(s)\) and write \(s=a(m,\lambda)\), where \(a\) is semialgebraic. This substitution transforms \(h_m'(s)\,ds\) into \(d\lambda\). Since \(0\le h_m\le L_U\), the integration range of \(\lambda\) lies in \([0,L_U]\).
Minkowski's inequality leads us to consider kernels of the form
\[
K_\lambda(x,y)
=\frac{1}{\mi(x-y)\sqrt{1+(x+y)^2}}
e^{\mi(x-y)\sqrt{1+(x+y)^2}a(x+y,\lambda)}.
\]
As before, we can remove the factor \(\sqrt{1+(x+y)^2}\), thus the remaining task is to establish uniform \(L^2\) estimates for kernels of the form
\[
\widetilde K_\lambda(x,y)
=\frac{1}{\mi(x-y)}
e^{\mi(x-y)\sqrt{1+(x+y)^2}a(x+y,\lambda)}
\]
which aligns with the result in Section \ref{reduction}. And we will prove such a result in Section \ref{oscillatory} and \ref{complex}. 
\subsection{Counterexample for a possible generalization}
Before the proof, we caution that if we replace the phase by a general semialgebraic phase, there is no uniform $L^2 \to L^2$ estimate depending only on the complexity. To see this, let \(I=[-1,1]\) and, for \(0<\epsilon<1/10\), define
\[
\Phi_\epsilon(x,y)
=\frac{x-y}{2\pi\sqrt{(x-y)^2+\epsilon^2}},
\;\;
T_\epsilon f(x)
=\int_I
\frac{e^{\mi\Phi_\epsilon(x,y)}}{x-y}f(y)dy.
\]
These phases are semialgebraic, and their graphs are described by
\[
4 \pi ^2 z^2\big((x-y)^2+\epsilon^2\big)=(x-y)^2,
\;\;z(x-y)\ge0.
\]
Thus their complexity is bounded independently of \(\epsilon\).

Taking \(f\equiv1\), we obtain for \(x\in[-1/2,1/2]\)
\[
\operatorname{Im}T_\epsilon1(x)
\ge
2\int_\epsilon^{1/2}
\frac{\sin\big(t/\sqrt{t^2+\epsilon^2}\big)}{t}\,dt
\gtrsim \log(1/\epsilon).
\]
Here the integrand defining the imaginary part is nonnegative, and
\(1/\sqrt2\le t/\sqrt{t^2+\epsilon^2}\le1\) for \(t\ge\epsilon\). Consequently,
\[
\|T_\epsilon\|_{L^2(I)\to L^2(I)}
\gtrsim \log(1/\epsilon).
\]

\bigskip
\section{Reduction to oscillatory integrals}
\label{reduction}
Now we want to do similar reductions as Section \ref{example}, so we should rewrite the phase as \[\mi \big(\xi(x-y)+\eta(\gamma(x)-\gamma(y)\big)=\mi (x-y) (\xi+\eta \cdot \frac{\gamma(x)-\gamma(y)}{x-y}).
\]

We now define
\[
\varphi(x,y):=
\begin{cases}
\dfrac{\gamma(x)-\gamma(y)}{x-y} & \mathrm{~if ~}x\neq y,\\[6pt]
\gamma'(x) & \mathrm{~if ~}x=y.
\end{cases}
\]

Then $\varphi\in C^1([-1,1]^2)$. Since \(\gamma \in C^2([-1,1])\), it is easy to obtain this for \(x \neq y\), and for \(x=y\), differentiating in the integral sign 
\[
\varphi(x,y)=\int_0^1 \gamma'\bigl(tx+(1-t)y\bigr)\,dt.
\]

We have
\[
\varphi_x(x,x)=\varphi_y(x,x)
=\frac12\gamma''(x).
\]

Now 

\begin{align*}
    K(x,y) &=\int_U e^{\mi \big(\xi(x-y)+\eta(\gamma(x)-\gamma(y)\big)}d\xi d\eta\\
    &= \int_U e^{\mi(x-y) \big(\xi+\eta\cdot\varphi(x,y)\big)}d\xi d\eta.
\end{align*}

Doing change of variables \((\xi,\eta) \mapsto (c,\eta)=(\xi+\varphi(x,y)\eta,\eta)\), then 
\begin{align*}
    K(x,y)&= \int 1_U(c-\varphi(x,y)\eta,\eta) e^{\mi c(x-y)}dc d\eta\\
    &= \int e^{\mi c(x-y)} H(\varphi(x,y),c)dc
\end{align*}
where \(H(m,c)\) is defined as \(\left|\{\eta:(c-m\eta,\eta)\in U\}\right|\). In fact, if we define
\[
l_{m,c}=\{(\xi,\eta):\xi+m\eta=c\},
\]
then 
\[
H(m,c)= \frac{|l_{m,c}\cap U|}{\sqrt{1+m^2}}\le L_U.
\]
\begin{lemma}
\(H(m,c)\) is a semialgebraic function.    
\end{lemma}

\begin{proof}
    Consider \(S=\{(m,c,\eta):(c-m\eta,\eta)\in U\}\), then \(S\subset \R^3\) is a semialgebraic set since we can simply plug \((c-m\eta,\eta)\) into the defining polynomials of \(U\). Now \(H(m,c)=\left|\{\eta:(c-m\eta,\eta)\in U\}\right|=\left| \{\eta:(m,c, \eta) \in S\}\right|\) is the slicing function of \(S\), and hence is semialgebraic by Theorem \ref{slice}.
    
\end{proof}

Heuristically, when \(|x-y|\) is large, we can obtain some decay with a factor \(\dfrac{1}{|x-y|}\), while when \(|x-y|\) is small, such estimates can be ineffective. 

An easy upper bound is \[
|K(x,y)|\le \int_U 1 \,d\xi d\eta= \mathrm{Area}(U),
\]
and in fact when \(x=y\) we can't obtain anything better than that. The following theorem gives an estimate of Area\((U)\), from which we can deal with the case when \(|x-y|\) is small.
\subsection{Near diagonal estimate}\label{neardiagsubsec}
\begin{theorem}
\label{area}
   Assume that \(U \subset \R^2\) is a semialgebraic set with complexity $k$, then we have \eqref{AreaboundedbyL2}, i.e.
   \[
   \mathrm{Area}(U) \lesssim_k L_U^2.
   \]
\end{theorem}

By Theorem \ref{cell}, we can decompose \(U\) into \(\lesssim_k 1\) \(C^2\)-cells. Notice that only \((1,1)\)-cells will contribute to the area of \(U\), so we only need to prove for a single \((1,1)\)-cell. Hence without loss of generality, we assume
\[
U=\{(\xi,\eta): \xi \in (a,b),f(\xi)\le \eta \le f(\xi)+h(\xi)\}
\]
where \(h>0\), and \(f, h \in C^2(a,b)\). By definition, we have \(\|h\|_{L^{\infty}}\le L_U\).

To prove Theorem \ref{area}, we need the following lemma.
 \begin{lemma}
    \label{curvature}
    Let \(W=[\xi_0-r,\xi_0+r]\subset(a,b)\). Suppose that \(\forall \xi\in W\), 
    \[ h(\xi)>t \mathrm{~and~}|f''(\xi)|<\frac{t}{2r^2}.\]
    Then
    \[L_U \ge 2r\sqrt{1+f'(\xi_0)^2}.\]
    \end{lemma}
    \begin{proof}[Proof of Lemma \ref{curvature}]
       Consider
       \[\ell(\xi)= f(\xi_0)+f'(\xi_0)(\xi-\xi_0)+\frac{t}{2}.\]
       Notice that \(\forall \xi\in[\xi_0-r,\xi_0+r]\), there exists \(\zeta\in[\xi_0-r,\xi_0+r]\) such that
       \[
       f(\xi)=f(\xi_0)+f'(\xi_0)(\xi-\xi_0)+\frac{1}{2} (\xi-\xi_0)^2f''(\zeta).\]
       Since \[
       \left|\frac{1}{2}(\xi-\xi_0)^2f''(\zeta)\right|\le \frac{1}{2}r^2\cdot\frac{t}{2 r^2}=\frac{t}{4},
       \]
       we have 
       \[
       f(\xi_0)+f'(\xi_0)(\xi-\xi_0)-\frac{t}{4}\le f(\xi)\le f(\xi_0)+f'(\xi_0)(\xi-\xi_0)+\frac{t}{4},
       \]
       and hence \(\forall \xi \in W\),
       \[
       f(\xi)<\ell(\xi)<f(\xi)+t\le f(\xi)+h(\xi).
       \]
       So \(\{(\xi,\ell(\xi)):\xi \in (a,b)\}\subset U\) and $ L_U \ge 2r\sqrt{1+f'(\xi_0)^2}$.
       
       \end{proof}

The lemma above says that if the set is sufficiently flat, in the sense that the second derivative is controlled, then it behaves like square and we can find a long line segment inside it. Otherwise, the set bends more rapidly and resembles an annulus, in which case we can obtain an estimate using curvature. It is also worth mentioning that both squares and annulus are sharp examples for the area estimate.
\begin{proof} [Proof of Theorem \ref{area}]

\begin{align*}
    \mathrm{Area}(U)&=\int_a^b h(\xi) d\xi= \|h\|_{L^1}\\
    &=\int_0^{+\infty}\left|\{\xi:|h(\xi)|>t\}\right| dt\\
    &=\int_0^{L_U}\left|\{\xi:h(\xi)>t\}\right|dt.
\end{align*}
Denote
\[
A_t:=\{\xi:h(\xi)>t\} \;,\;
B_t:=\left\{\xi:\frac{|f''(\xi)|}{1+f'(\xi)^2}>\frac{t}{6L_U^2}\right\}.
\]
Then \[
\mathrm{Area}(U)=\int_0^{L_U} |A_t| dt.
\]
and now we should give an upper bound of \(|A_t|\).
   
 By Theorem \ref{component}, \(A_t,B_t\) are some finite union of open intervals. 
 And \(\forall t\in[0,L_U]\), \(\exists M\lesssim_k 1\) s.t. the number of intervals in \(A_t\) and \(B_t\) \(\le M\).

As long as \(A_t\neq \emptyset\), we have \(L_U \ge t\). Assume
       \[A_t=\bigcup_{j=1}^N(a_j,b_j).\]
If \(b_j-a_j>2L_U\), let \(I_j=(a_j+L_U,b_j-L_U)\), then
       \[|A_t|\le 2L_UM+\sum_{j=1}^N|I_j|.
       \]
       
 For \(\forall \xi_0\in I_j\), denote 
       \[
       r_0=\frac{L_U}{\sqrt{1+f'(\xi_0)^2}}. \]
       
 If \(\forall \xi \in [\xi_0-r_0,\xi_0+r_0]
 \), \(|f''(\xi)|< \dfrac{t}{2r_0^2}\), we obtain a contradiction by applying Lemma \ref{curvature} and get
       \[
       L_U\ge 2r_0\sqrt{1+f'(\xi_0)^2}=2L_U.\]

 Then \(\exists \xi \in[\xi_0-r_0,\xi_0+r_0]\), such that
       \(|f''(\xi)|\ge \dfrac{t}{2r_0^2}\). If \(|f''(\xi_0)|\ge \dfrac{t}{2r_0^2}\), then we directly have \[
       \frac{|f''(\xi_0)|}{1+f'(\xi_0)^2}\ge \frac{t}{2L_U^2}.
       \]

  Now we can assume \(|f''(\xi_0)|< \dfrac{t}{2r_0^2}\), and choose \(\xi_1\) closest to \(\xi_0\), such that \(\xi_1 \in [\xi_0-r_0,\xi_0+r_0]\) and \(|f''(\xi_1)|= \dfrac{t}{2r_0^2}\). Let \(J=[\min\{\xi_0,\xi_1\},\max\{\xi_0,\xi_1\}]\), then \(\forall \eta \in J\), \(|f''(\eta)|\le \dfrac{t}{2r_0^2}\). Now we have
 
       \[
       |f'(\xi_1)-f'(\xi_0)|\le r_0\sup_{\eta 
       \in J}|f''(\eta)|\le \frac{t}{2r_0}.\]
       
 Hence
       \begin{align*}
       1+f'(\xi_1)^2 & \le 1+\left(|f'(\xi_0)|+\frac{t}{2r_0}\right)^2\le 1+2f'(\xi_0)^2+\frac{t^2}{2r_0^2}\\
       &\le 2\left(1+f'(\xi_0)^2\right)+ \frac{L_U^2}{2r_0^2}<\frac{3L_U^2}{r_0^2}
       \end{align*}

and 
\[
\frac{|f''(\xi_1)|}{1+f'(\xi_1)^2}>\frac{t}{6L_U^2}.
\]
       
 Hence \(I_j \subset N_{L_U} (B_t)\), and \(|A_t|\lesssim  L_UM+ |B_t|\).
       \[
       \mathrm{Area}(U)=\int_0^{L_U}|A_t|dt\lesssim_k L_U^2+\int_0^{L_U}|B_t|dt.\]
       
 Let \(u=\frac{t}{6L_U^2}\), then
       \begin{align*}
       \int_0^{L_U}|B_t|dt&=\int_0^{L_U}\left|\left\{\xi:\frac{|f''(\xi)|}{1+f'(\xi)^2}>\frac{t}{6L_U^2}\right\}\right|dt\\
       &=6L_U^2\int_0^{\frac{1}{6L_U}}\left|\left\{\xi:\frac{|f''(\xi)|}{1+f'(\xi)^2}>u\right\}\right|du\\
       &\le 6L_U^2\int_0^{+\infty}\left|\left\{\xi:\frac{|f''(\xi)|}{1+f'(\xi)^2}>u\right\}\right|du\\
       &=6L_U^2\left\|\frac{f''(\cdot)}{1+f'(\cdot)^2}\right\|_{L^1}.
       \end{align*}

 Notice that for \(g(\xi)=\arctan(f'(\xi))\), we have
       \[g'(\xi)=\frac{f''(\xi)}{1+f'(\xi)^2}\]
       and \(g'\) changes sign \(\lesssim_k 1\) times. Hence 

       \[\int_0^{L_U}|B_t|\,dt\le 6L_U^2\|g'\|_{L^1}=6L_U^2\operatorname{Var}(g)\lesssim_k L_U^2.\]
       
\end{proof}
Now by Theorem \ref{area} and Schur's test, the operator associated with kernel \[K(x,y)1_{|x-y|<\frac{1}{L_U}}\]
is \(L^2\) bounded by $L$ (with implicit constant depending on $K$), since
\[
\int_{\{y:|x-y|<\frac{1}{L_U}\}} |K(x,y)|dy\le U \cdot  \frac{1}{L_U} \lesssim_k L_U.
\]
\subsection{Off diagonal estimate}\label{offdiagsubsec}
Now we are left to consider the kernel 
\[
 K(x,y)= 1_{\{|x-y|>\frac{1}{L_U}\}}\int e^{\mi c(x-y)} H(\varphi(x,y),c)dc.
\]
Denote \(m:=\varphi(x,y)\), notice that for all $x,y$ and $c$, \(0\le H(m,c)\le L_U\), so 
\[
H(m,c)=\int_{0}^{L_U} 1_{\{\lambda<H(m,c)\}} d\lambda
\]
and \[
 K(x,y)= \int_{0}^{L_U}\left(1_{\{|x-y|>\frac{1}{L_U}\}}\int e^{\mi c(x-y)}1_{\{\lambda<H(m,c)\}}dc\right)\,d\lambda.
\]

Denote \[
K_\lambda(x,y)=1_{\{|x-y|>\frac{1}{L_U}\}}\int e^{\mi c(x-y)}1_{\{\lambda<H(m,c)\}}dc.
\]

Then \[
K(x,y)=\int_{0}^{L_U} K_\lambda(x,y) d\lambda
\]
and \[
|Tf(x)|=\left|\int_{\R}\int_{0}^{L}K_\lambda(x,y)f(y)d\lambda dy\right|\le \int_{0}^{L} \left| \int_{\R}K_\lambda(x,y)f(y)dy\right|d\lambda.
\]

By Minkowski's inequality, 
\[
\|Tf\|_{L^2}=\|\int_{\R}K_\lambda(x,y)f(y)dy\|_{L_x^2L_\lambda^1}\le \|\int_{\R}K_\lambda(x,y)f(y)dy\|_{L_\lambda^1L_x^2}.
\]

Hence it suffices to show that for any fixed \(\lambda\),
\[
\|\int_{\R}K_\lambda(x,y)f(y)dy\|_{L_x^2}\lesssim_k \|f\|_{L^2}.
\]

For each fixed \(\lambda\), \(E_\lambda:= \{(m,c): H(m,c)>\lambda\}\) is semialgebraic, since it is the superlevel set of \(H(m,c)\). We can do cell decomposition with regard to \(m\). To be precise, 
there exists \(I_1, \cdots, I_N \subset \R\), and for each \(1 \le j \le N\), there exists semialgebraic function \(a_{\lambda,j}\le b_{\lambda,j}\) defined on \(I_j\), such that up to a set of two-dimensional Lebesgue measure zero
\[
E_\lambda=\bigcup_{j=1}^{N} \{(m,c) \in \R^2, m\in I_j,a_{\lambda,j}(m)< c  < b_{\lambda,j} (m)\}
\]
where \(N\) and the complexity of \(a_{\lambda,j}, b_{\lambda,j}\) \(\lesssim_k 1\).

Thus 
\begin{align*}
    K_\lambda(x,y)&=1_{\{|x-y|>\frac{1}{L_U}\}}\int_{\{c: (m,c)\in E_\lambda\}} e^{\mi c(x-y)}dc\\
    &= \sum_{j=1}^{N} 1_{\{|x-y|>\frac{1}{L_U}\}} 1_{\{m \in I_j\}} \int_{a_{\lambda,j}(m)}^{b_{\lambda,j} (m)} e^{\mi c(x-y)}dc\\
    &= 1_{\{|x-y|>\frac{1}{L_U}\}} \sum_{j=1}^{N} 1_{\{m\in I_j\}}\frac{e^{\mi b_{\lambda,j} (m)(x-y)}-e^{\mi a_{\lambda,j} (m)(x-y)}}{\mi(x-y)} .
\end{align*}

Then it suffices to show that for any interval \(I \subset \R\), and semialgebraic function \(g: I \to \R\) with complexity \(\lesssim_k 1\), the operator associated with kernel 
\[
1_{\{|x-y|>\frac{1}{L_U}\}} 1_{\{\varphi(x,y) \in I\}}\frac{e^{\mi g(\varphi(x,y))(x-y)}}{x-y} 
\]
is \(L^2\) bounded by \(\lesssim_k 1\).

\bigskip
\section{Estimates of the oscillatory integrals}\label{oscillationsec}
\label{oscillatory}
\subsection{Setup and Preliminary Reduction}

In the next sections, we prove the following theorem.

\begin{theorem}\label{main}
Let $I\subset[-2,2]$ be an interval and let $g:I\to\R$
be measurable. For $\epsilon>0$, define
\[
K_g(x,y):=
1_{\{|x-y|>\epsilon\}}
1_{\{\varphi(x,y)\in I\}}
\frac{e^{\mi (x-y)g(\varphi(x,y))}}{x-y}
\]
and
\[
T_gf(x):=\int_{[-1,1]} K_g(x,y)f(y)\,dy.
\]

If $g$ is semialgebraic with complexity $\lesssim_k 1$, then
\[
\|T_gf\|_{L^2([-1,1])}
\lesssim_k \|f\|_{L^2([-1,1])},
\]
uniformly in $\epsilon>0$.
\end{theorem}

For the same $I$, $g$, and $\epsilon$, define
\[
\bar K_g(x,y):=
1_{\{|x-y|>\epsilon\}}
1_{\{\gamma'(x)\in I\}}
1_{\{\gamma'(y)\in I\}}
\frac{e^{\mi (x-y)g(\varphi(x,y))}}{x-y}
\]
and let $\bar T_g$ denote the corresponding integral operator
on $[-1,1]$. We will show that the uniform $L^2$ boundedness of $T_g$ is equivalent to
that of $\bar T_g$. Each formulation has its own advantages, so we will switch between them as needed in what follows.

\begin{lemma}\label{hilbert}
The operator
\[
Cf(x):=\int_0^\infty \frac{f(y)}{x+y}\,dy,
\qquad x>0
\]
is bounded on $L^2(0,\infty)$.  Also see  \cite[Chapter~9]{HardyLittlewoodPolya1952}.
\end{lemma}

\begin{proof}
Changing variables $y=xt$ and applying Minkowski's inequality,
we obtain
\[
\begin{aligned}
\|Cf\|_2
&\le
\int_0^\infty \frac{\|f(t\,\cdot)\|_2}{1+t}\,dt\\
&=
\|f\|_2\int_0^\infty \frac{t^{-1/2}}{1+t}\,dt
=
\pi\|f\|_2.
\end{aligned}
\]
\end{proof}

\begin{proposition}
\label{kernel-equivalence}
The $L^2$ boundedness of $T_g$ is equivalent to
that of $\bar T_g$.
\end{proposition}

\begin{proof}
We may assume $I=(\alpha,\beta)$, since changing the endpoints
only affects null sets. Set
\[
\mathcal C:=
\{c\in(-1,1):\gamma'(c)\in\{\alpha,\beta\}\}.
\]

Then \(\mathcal C\) is finite. And since $\gamma'$ is strictly increasing and $\varphi(x,y)$
lies between $\gamma'(x)$ and $\gamma'(y)$, we have
\[
|K_g(x,y)-\bar K_g(x,y)|
\le
\sum_{c\in\mathcal C}
\frac{1_{\{x<c<y\}}+1_{\{y<c<x\}}}{|x-y|}.
\]

 For each $c\in\mathcal C$,
the change of variables $x=c-s$, $y=c+t$ identifies the kernel
$1_{\{x<c<y\}}/|x-y|$ with the kernel $1/(s+t)$,
where $s,t>0$. The other term is treated similarly.
Thus Lemma~\ref{hilbert} gives
\[
\|(T_g-\bar T_g)f\|_{L^2([-1,1])}
\lesssim \|f\|_{L^2([-1,1])}.
\]

\end{proof}

\subsection{Main ideas for dealing with the phase}\label{thirdordersubsec}
Now we consider the oscillatory integrals in Theorem \ref{main}. The kernel \(K_g\) is the Calderon Zygmund kernel \(\frac{1_{|x-y|>\epsilon}}{x-y}\) times an oscillatory factor with phase \(\mi (x-y)g(\varphi(x,y))\). If the phase function can be written as \(\mi (F(x)-G(y))\), then
\[
T_gf(x)=e^{\mi F(x)}\int \frac{1}{x-y} (e^{-\mi G(y)}f(y)) dy
\]
and the \(L^2\) boundedness follows from (modulated) Hilbert transform.

Unfortunately, in general we do not have this clean form of $\mi (F(x)-G(y))$ for the Kernel. An important method we adapt throughout our proof is that in Ricci-Stein's paper \cite{RicciStein1987}. We begin with a brief outline of the proof strategy. When \(|x-y|\) is small, we approximate the phase with the form \(\mi (F(x)-G(y))\); while when \(|x-y|\) is large, we do a dyadic decomposition and \(TT^*\) method again, and use van der Corput method to obtain the decay.  As in the previous example, the estimates away from the
diagonal $x=y$ do not rely on cancellation between paired points.

We now discuss the phase approximation, as in \cite{RicciStein1987}.
For simplicity, we assume that $g$ is smooth, and set
\[
P(x,y):=(x-y)g(\varphi(x,y)).
\]

Notice that $P(x,y)$ can be represented in the form $F(x)-G(y)$
if and only if $P_{xy}=0$, so we first calculate $P_{xy}$ and obtain
\begin{align*}
P_{xy}(x,y)
&=(x-y)g''(\varphi(x,y))\varphi_x(x,y)\varphi_y(x,y)\\
&\quad
+g'(\varphi(x,y))
\bigl(\varphi_y(x,y)-\varphi_x(x,y)
+(x-y)\varphi_{xy}(x,y)\bigr).
\end{align*}
In fact, this expression can be simplified further, since
\begin{align*}
&\varphi_y(x,y)-\varphi_x(x,y)+(x-y)\varphi_{xy}(x,y)\\
&\qquad=\partial_x\partial_y\bigl((x-y)\varphi(x,y)\bigr)\\
&\qquad=\partial_x\partial_y\bigl(\gamma(x)-\gamma(y)\bigr)=0.
\end{align*}
Hence
\[
P_{xy}(x,y)
=(x-y)g''(\varphi(x,y))\varphi_x(x,y)\varphi_y(x,y),
\]
which depends on $g$ only through its second derivative, and this
is crucial to our whole proof.

Now we cannot in general expect to represent $P$ without an error
term, so we have to consider an approximation. A simple approximation is
\[
P(x,y)
=xg(\gamma'(x))-yg(\gamma'(y))+O(|x-y|)
\]
locally near the diagonal. To improve the order of the error, we seek
a decomposition
\[
P(x,y)=F(x)-G(y)+E(x,y)
\]
such that
\[
E(x,x)=E_x(x,x)=E_y(x,x)=0.
\]
Differentiating this identity with respect to $x$ and $y$, we have
\begin{align*}
P_x(x,y)&=F'(x)+E_x(x,y),\\
P_y(x,y)&=-G'(y)+E_y(x,y).
\end{align*}
Since $P(x,x)=0$, the condition $E(x,x)=0$ implies $F=G$, and
evaluating the two identities above on the diagonal gives
\[
F'(x)=G'(x)=g(\gamma'(x)).
\]
This choice is unique up to a common additive constant, and for a fixed
base point $x_0$ we may take
\[
F(x)=G(x)=\int_{x_0}^x g(\gamma'(t))\,dt.
\]

Now we estimate the error term $E(x,y)$ and see what order in
$|x-y|$ we can obtain, using a Taylor expansion about the diagonal:
\begin{align*}
E(x,y)
&=E(x,x)+(y-x)E_y(x,x)
+\frac{(y-x)^2}{2}E_{yy}(x,x)\\
&\quad+\frac{(y-x)^3}{6}E_{yyy}(x,x)+O(|x-y|^4).
\end{align*}
We already know that $E(x,x)=E_y(x,x)=0$, so it remains to calculate
the higher derivatives on the diagonal, and we start from the identity
\[
E_{xy}(x,y)=P_{xy}(x,y)
=(x-y)g''(\varphi(x,y))\varphi_x(x,y)\varphi_y(x,y).
\]
Differentiating $E_y(x,x)=0$ along the diagonal, we have
\[
0=\frac{d}{dx}E_y(x,x)
=E_{xy}(x,x)+E_{yy}(x,x),
\]
and hence
\[
E_{yy}(x,x)=-E_{xy}(x,x)=0.
\]
Now differentiating $E_{yy}(x,x)=0$ once more along the diagonal,
we obtain
\[
E_{xyy}(x,x)+E_{yyy}(x,x)=0.
\]
On the other hand, using
$\varphi_x(x,x)=\varphi_y(x,x)=\frac{1}{2}\gamma''(x)$, we have
\begin{align*}
E_{xyy}(x,x)
&=\left.
\partial_y\bigl((x-y)g''(\varphi(x,y))
\varphi_x(x,y)\varphi_y(x,y)\bigr)
\right|_{y=x}\\
&=-g''(\varphi(x,x))\varphi_x(x,x)\varphi_y(x,x)\\
&=-\frac{1}{4}g''(\gamma'(x))\bigl(\gamma''(x)\bigr)^2.
\end{align*}
Hence
\[
E_{yyy}(x,x)
=\frac{1}{4}g''(\gamma'(x))\bigl(\gamma''(x)\bigr)^2,
\]
and substituting this into the Taylor expansion gives
\[
E(x,y)
=-\frac{1}{24}g''(\gamma'(x))\bigl(\gamma''(x)\bigr)^2(x-y)^3
+O(|x-y|^4).
\]
Thus, when $g''(\gamma'(x))(\gamma''(x))^2\neq0$, the correct order
to expect is $3$, and this is what we will make explicit in what follows.
\subsection{Estimates under uniform derivative bounds}
Now we first deal with the case where \(|g''|\) has a uniform lower bound, in particular since \(g\in C^3(I)\), \(g''\) is positive or negative on \(I\). To be precise, we have the following lemma.
\begin{lemma}
\label{ricci-stein}
    Let \(I=(a,b)\subset [-2,2]\) be an interval, \(g \in C^3(I)\). Assume that for \(\kappa>0\) and \(K\ge 1\), \(\forall t \in I\), we have
    \[
    \kappa \le |g''(t)| \le K\kappa, \;\; |g'''(t)|\le K\kappa.
    \]
    Then 
    \[
    \|\bar T_g f\|_{L^2} \lesssim_K \|f\|_{L^2}
    \mathrm{~and ~equivalently ~} \|T_g f\|_{L^2} \lesssim_K \|f\|_{L^2}.\]
   
\end{lemma}
\begin{proof}
Since \(\gamma'\) is strictly increasing, the set
\(J:=\{x:\gamma'(x)\in I\}\)
is an interval \(J=(\alpha,\beta)\). Then
\[
1_{\{\gamma'(x)\in I\}}1_{\{\gamma'(y)\in I\}}
=1_{\{x\in J\}}1_{\{y\in J\}}.
\]
As in our discussion above, let \[
G(x)=\int_{\alpha}^{x}g(\gamma'(t))\,dt.
\]
Then \[
(x-y)g(\varphi(x,y))-(G(x)-G(y))=\int_{y}^{x}g(\varphi(x,y))-g(\gamma'(t))dt.
\]

By symmetry, we may assume that \(x>y\) in the following calculations. Notice that \(\forall t\in (y,x)\), \(\exists u(t) \in I\), such that 
\[
g(\gamma'(t))=g(\varphi(x,y))+ (\gamma'(t)-\varphi(x,y))g'(\varphi(x,y))+\frac{(\gamma'(t)-\varphi(x,y))^2}{2} g''(u(t)),
\]
and 
\[
\int_{y}^{x} \gamma'(t)-\varphi(x,y)dt=\gamma(x)-\gamma(y)-(x-y)\varphi(x,y)=0.
\]

Then we have \[
\Big|(x-y)g(\varphi(x,y))-(G(x)-G(y))\Big|\lesssim \sup_{t\in I}|g''(t)|\int_{y}^{x} |\gamma'(t)-\varphi(x,y)|^2 dt.
\]

By the further estimate
\begin{align*}
    &\int_{y}^{x} |\gamma'(t)-\varphi(x,y)|^2 dt\\=&\int_{y}^{x}\left(\frac{|\gamma(x)-\gamma(y)-(x-y)\gamma'(t)|}{|x-y|}\right)^2 dt\\
    \le & \frac{1}{|x-y|}\sup_{t \in [y,x]} \left(\int_{y}^{x}|\gamma'(u)-\gamma'(t)|du\right)^2\\
    \lesssim &\frac{1}{|x-y|}\sup_{t \in [y,x]} \left(\int_{y}^{x}|u-t| du\right)^2
    \lesssim |x-y|^3,
\end{align*}
we obtain 
\[
\Big|(x-y)g(\varphi(x,y))-(G(x)-G(y))\Big|\lesssim K\kappa |x-y|^3.                                                                       
\]
    
Let \(r=\kappa^{-\frac{1}{3}}\), then for  \(\epsilon\le |x-y|\le r\), write

    \[
   \bar K_{g}(x,y)= 1_{\{\epsilon\le |x-y|\le r\}}1_{\{x\in J\}}1_{\{y\in J\}} \frac{e^{\mi(G(x)-G(y))}}{x-y}+K_{error}(x,y).
    \]
    
    The former term is easy to deal with by our modulation argument above, while for the error term, because of the fact that \(|e^{ix}-1|\le |x|\), we have
    \[
    |K_{error}(x,y)|\lesssim K \kappa|x-y|^2 1_{|x-y|\le r} .
    \] 
    
    By Schur's test, 
    \[\|T_{error}\|_{L^2\to L^2}\le \sup_{x}\int_{\{y: |x-y|\le r\}} K \kappa|x-y|^2dy \lesssim K \kappa r^3 \le K.\]

Now we are left to deal with the kernel
    \[
    1_{\{|x-y|>r\}}1_{\{x\in J\}}1_{\{y\in J \}} \frac{e^{\mi P(x,y)}}{x-y}.\]
    
    For the remaining part \(|x-y|> r\), we first consider \(x>y\), since the part where \(x<y\) can be treated in the same way. We decompose that into dyadic intervals \(\delta<x-y<2\delta\), where \(\delta=2^j r\) for some \(j\). Notice that here we only consider the one-sided interval, forcing \(x>y\), which will be used in later estimates.
    Then the corresponding operator is 
    \[
    \bar T_{g,\delta}f(x)= \int 1_{\{\delta<x-y<2\delta\}}1_{\{x\in J\}}1_{\{y\in J \}} \frac{e^{\mi P(x,y)}}{x-y} f(y) dy.
    \]
    
    Now we consider \(\bar T_{g,\delta}\bar T_{g,\delta}^{*}\), which is associated with kernel
    \[
    \widetilde{K}_\delta(x,z)=\int_{\{\delta<x-y<2\delta,\delta<z-y<2\delta\} }  \frac{e^{\mi (P(x,y)-P(z,y))}}{(x-y)(z-y)}dy.
    \]
    Here we suppress the factors \(1_{\{x\in J\}}1_{\{z\in J\}}\), and the integral in \(y\) is taken over
\[
J\cap\bigl(\max\{x,z\}-2\delta,\min\{x,z\}-\delta\bigr).
\]
If this interval is empty, then \(\widetilde K_\delta(x,z)=0\). Otherwise, denote it by \((A,B)\). Notice that \(|x-z|<\delta\) and \(B-A\le\delta\).
    
    Also denote\[
    \Phi(y)=P(x,y)-P(z,y)=(x-y)g(\varphi(x,y))-(z-y)g(\varphi(z,y))
    \]
    and \[
    b(y)=\frac{1}{(x-y)(z-y)}.
    \]

    Then \[
    \widetilde{K}_\delta(x,z)=\int_{A}^{B} e^{\mi \Phi(y)} b(y) dy.
    \]

    We want to use integration by parts here, then we first need to check \(\Phi'(y)\) doesn't vanish.
    \[
    \Phi'(y)=\int_{z}^{x}\partial_t\partial_yP(t,y)dt=\int_{z}^{x} (t-y)g''(\varphi(t,y))\varphi_{x}(t,y)\varphi_{y}(t,y)dt.
    \]
    
    Notice that
\[
\varphi_x(t,y)=\int_0^1 u\,\gamma''\bigl(y+u(t-y)\bigr)\,du.
\]
Since \(1\le\gamma''\le2\), we have \(\frac{1}{2}\le\varphi_x(t,y)\le1\),
and the same bounds hold for \(\varphi_y(t,y)\). 
Hence \(\forall t \in [z,x]\) (or \([x,z]\)), \((t-y)g''(\varphi(t,y))\varphi_{x}(t,y)\varphi_{y}(t,y)\) is positive or negative all the time, since we assume \(x>y\), \(z>y\) and \(g''\) doesn't change sign.

    Hence
    \[
    |\Phi'(y)|\gtrsim |x-z| \inf_{t} |(t-y)g''(\varphi(t,y))| \ge \kappa \delta |x-z|.
    \]

    Now we can do integration by parts as long as \(|x-z|\neq 0\),
    \[
     \widetilde{K}_\delta(x,z)=\frac{1}{\mi} \left[e^{\mi \Phi(y)}\frac{b(y)}{\Phi'(y)}\Big|_{A}^{B}-\int_{A}^{B}e^{\mi \Phi(y)}\frac{d}{dy}\left(\frac{b(y)}{\Phi'(y)}\right) dy\right].
    \]
    
    For the first term, 
    \[
    \left| \frac{b(y)}{\Phi'(y)}
    \right|\lesssim \frac{1}{\kappa \delta |x-z|}\cdot \frac{1}{\delta^2}=\frac{1}{\kappa \delta^3 |x-z|}.
    \]

    While for the second term,
    \[
    \left| \int_{A}^{B}e^{\mi \Phi(y)}\frac{d}{dy}\left(\frac{b(y)}{\Phi'(y)}\right) dy\right| \le \int_{A}^{B} \left|\frac{b'(y)}{\Phi'(y)}\right|dy+\int_{A}^{B} \left|\frac{b(y)\Phi''(y)}{(\Phi'(y))^2}\right|dy.
    \]

  For \(\forall y \in [A,B]\), \(b'(y)\) doesn't change sign, so 
    \[
    \int_{A}^{B} \left|\frac{b'(y)}{\Phi'(y)}\right|dy\lesssim \frac{1}{\kappa \delta|x-z|}\left|\int_{A}^{B} b'(y)dy\right|\lesssim \frac{1}{\kappa \delta^3 |x-z|},
    \]
    and by a similar argument as above,
\[
 |\Phi''(y)|= \left|\int_{z}^{x}\partial_t\partial_{yy}P(t,y)\,dt\right|
 \lesssim K\kappa |x-z|.
\]

In fact 
\begin{align*}
     \partial_t\partial_{yy}P(t,y)&=
(t-y)g'''(\varphi(t,y))
        \varphi_x(t,y)\varphi_y(t,y)^2\\
        &-g''(\varphi(t,y))
        \varphi_x(t,y)\varphi_y(t,y)\\
 &+(t-y)g''(\varphi(t,y))
        \bigl(\varphi_{xy}(t,y)\varphi_y(t,y)
        +\varphi_x(t,y)\varphi_{yy}(t,y)\bigr).
\end{align*}
And we claim that \(| \partial_t\partial_{yy}P(t,y)|\lesssim K\kappa\). Since \(t,y\in J\), \(\gamma'(J)\subset I\subset[-2,2]\), and \(\gamma''\ge1\), we have \(|t-y|\le|\gamma'(t)-\gamma'(y)|\lesssim 1\).
Moreover, we have already proved \(|\varphi_x|\lesssim 1\) and \(|\varphi_y|\lesssim 1\), and now we are left to show \(|(t-y)\varphi_{xy}(t,y)|\lesssim 1\) and \(|(t-y)\varphi_{yy}(t,y)|\lesssim 1\). This is obtained by considering \(\partial_{xy}((x-y)\varphi(x,y))\) and  \(\partial_{yy}((x-y)\varphi(x,y))\).

    Hence 
    \[
    \int_{A}^{B} \left|\frac{b(y)\Phi''(y)}{(\Phi'(y))^2}\right|dy\le \delta \sup_{y\in [A,B]}\left|\frac{b(y)\Phi''(y)}{(\Phi'(y))^2}\right|\lesssim \frac{K}{\kappa \delta^3 |x-z|}.
    \]

    Putting the estimates together, we have when \(|x-z| \neq 0\), 
    \[
     |\widetilde{K}_\delta(x,z)|\le C_K\cdot  \frac{1}{\kappa \delta^3 |x-z|}.
    \]

    And we also have the trivial estimate\[
    |\widetilde{K}_\delta(x,z)|\le \delta \sup_{y\in [A,B]} |b(y)|\le \frac{1}{\delta}.
    \]

    By Schur's test,
  \begin{align*}
      \|\bar T_{g,\delta}\|_{L^2 \to L^2}^2 &\le \sup_{x}\int_{\{z: |x-z|\le \delta\}}  |\widetilde{K}_\delta(x,z)|dz \\
    &\lesssim_K \sup_{x}\int_{\{z:|x-z|\le (\kappa \delta^2)^{-1}\}} \frac{1}{\delta} dz+ \int_{\{z: (\kappa \delta^2)^{-1}\le |x-z|\le \delta \}} \frac{1}{\kappa \delta^3 |x-z|} dz\\
    & \lesssim_K \frac{1+\ln(\kappa \delta^3)}{\kappa\delta^3}\lesssim (\kappa\delta^3)^{-\frac{1}{2}}.
  \end{align*}

  Hence \[
   \|\bar T_{g,\delta}\|_{L^2 \to L^2}\lesssim_K (\kappa\delta^3)^{-\frac{1}{4}}.
  \]

  Take the sum over \(\delta\), and notice that \(\kappa r^3 = 1\), we get \(\|\bar T_g\|_{L^2\to L^2}\lesssim_K 1\), as desired.

\end{proof}
\subsection{General Case}
In Lemma \ref{ricci-stein}, we mainly dealt with the case \(|g''|\) has a uniform lower bound. For the general case, we want to rule out the points where \(g''=0\), and break the else parts into disjoint intervals. This is harmless as long as the number of intervals \(\lesssim_K 1\).

\begin{lemma}
\label{partition}
Suppose that \(g\) is a function defined on an interval \(I\), and $I$ is the disjoint union, up to endpoints, of
$\lesssim_K 1$ intervals $I_j$, and that
\[
\|T_{g,I_j}\|_{L^2\to L^2}\lesssim_K 1
\]
uniformly in $j$. Then
\[
\|T_{g,I}\|_{L^2\to L^2}\lesssim_K 1,
\]
where \(T_{g,I}\) is associated with 
\[
K_{g,I}(x,y):=1_{\{|x-y|>\epsilon\}}1_{\{\varphi(x,y)\in I\}} \frac{e^{\mi (x-y)g(\varphi(x,y))}}{x-y}.
\]
We use this notation when we need to specify the interval containing \(\varphi(x,y)\).
\end{lemma}

\begin{proof}
Notice that 
\[
T_{g,I}=\sum_{j}T_{g,I_j}.
\]
The conclusion follows from the triangle inequality, since the number of intervals is \(\lesssim_K 1\).

\end{proof}

After this partition, we can assume that on each interval \(I\), either \(g''\) doesn't vanish, or \(g''\equiv 0\). The latter is easy to deal with.
\begin{lemma}
    If \(g''\equiv0\) on \(I\), equivalently \(g(t)=at+b\), then \[\|T_{g}\|_{L^2\to L^2}\lesssim_ 1 .
    \]
\end{lemma}

\begin{proof}
 We only need to prove the \(L^2\) boundedness of \(\bar T_g\), which has the kernel
\begin{align*}
    \bar K_g(x,y)&=1_{\{|x-y|>\epsilon\}}1_{\{\gamma'(x)\in I \}}1_{\{\gamma'(y)\in I \}}
\frac{e^{\mi(x-y)(a\varphi(x,y)+b)}}{x-y}\\
& =1_{\{|x-y|>\epsilon\}}1_{\{\gamma'(x)\in I \}}1_{\{\gamma'(y)\in I \}}
\frac{e^{\mi(a(\gamma(x)-\gamma(y))+b(x-y))}}{x-y}.
\end{align*}

Then the \(L^2 \) boundedness follows from modulated Hilbert transform.

\end{proof}

We now focus on the case where \(g''\) does not vanish on \(I\), but may vanish at an endpoint, so that \(|g''|\) may not have the uniform positive lower bound required in Lemma \ref{ricci-stein}. The following lemma gives an estimate adapted to this situation when \(I=(0,L)\), using a dyadic decomposition near \(0\).

\begin{lemma}
\label{one-endpoint}
    Let \(I=(0,L)\), \(0<L\le 2\). And let \(g\in C^3(I)\), \(g''\neq 0\) on \(I\). If
    \[
    \left|t\frac{g'''(t)}{g''(t)}\right| \le K,
    \]
    then 
    \[
    \|\bar T_{g,I}\|_{L^2\to L^2}\lesssim_K 1\mathrm{~and ~equivalently~}\| T_{g,I}\|_{L^2\to L^2}\lesssim_K 1.
    \]
     
\end{lemma}
\begin{proof}
    Since $g''$ does not vanish,
\[
\left|\frac{d}{dt}\log|g''(t)|\right|
=
\left|\frac{g'''(t)}{g''(t)}\right|
\le \frac Kt.
\]

Thus, for $t_1<t_2$ in $I$,
\[
\left|\log\frac{|g''(t_2)|}{|g''(t_1)|}\right|
\le K\int_{t_1}^{t_2}\frac{dt}{t}
=K\log\frac{t_2}{t_1},
\]
and hence
\[
\left(\frac{t_1}{t_2}\right)^K
\le
\frac{|g''(t_2)|}{|g''(t_1)|}
\le
\left(\frac{t_2}{t_1}\right)^K.
\]

Let
\[
I_l=(2^l,2^{l+1})\cap I, \;\; g_l(t)= g\big|_{I_l}(t).
\]

Just to mention, the exponent \(l\) will mostly be negative. By the estimates above, we have
\[
\sup |g_l''|\le 2^K \inf |g_l''|.
\]

Now we want to apply Lemma \ref{ricci-stein}, and we already have the desired estimate for \(g''\) as long as we substitute \(K\) by \(2^K\), and just choose \(\kappa=\inf |g_l''|\). However we know nothing of \(|g_l'''|\) other than \[
\left|g_l'''(t)\right|\le \left|\frac{Kg_l''(t)}{t}\right|.
\]

The constant will lose control if \(t\) is close to zero. To deal with this, we want to rescale the interval \(I_l\) to \([1,2]\). To be precise, let \(J_l=2^{-l} I_l\subset (1,2)\), \(\forall t \in J_l\)
\[
G_l(t)=2^{l}g(2^lt).
\]

Then
\[
G_l''(t)=2^{3l}g''(2^lt),
\;\;
G_l'''(t)=2^{4l}g'''(2^lt).
\]

Taking
\[
\kappa_l=2^{3l}\inf_{I_l}|g''|>0,
\]

we obtain
\[
\kappa_l\le |G_l''(t)|\le 2^K\kappa_l.
\]

Furthermore, since $t\in (1,2)$,
\[
|G_l'''(t)|
\le \frac Kt|G_l''(t)|
\le K2^K\kappa_l.
\]

Thus Lemma \ref{ricci-stein} gives
\[
\| T_{G_l,J_l}F\|_{L^2}
\lesssim_K \|F\|_{L^2}
\]
as long as the corresponding curve \(\Gamma\) satisfies \(1\le \Gamma''\le 2\).

Notice that if  we assume \(u=2^l x\), \(v=2^l y\), \(F(y):=f(2^l y)\), 
\[\Gamma(x):=\frac{\gamma(2^lx)}{2^{2l}}\mathrm{~and~}\Phi(x,y) :=\frac{\Gamma(x)-\Gamma(y)}{x-y},
\]
then it's easy to check\[
\Gamma''(x)=\gamma''(2^l x)\in [1,2]
\]
and 
\begin{align*}
T_{g_l,I_l}
f(u)&=\int 1_{\{\varphi(u,v)\in I_l\}}1_{|u-v|>\epsilon} \frac{e^{\mi (u-v)g(\varphi(u,v))}}{u-v}f(v) dv\\
&=\int  1_{\{\Phi(x,y)\in J_l\}}
 1_{\{|x-y|>\frac{\epsilon}{2^l}\}}
\frac{e^{\mi (x-y)G_l(\Phi(x,y))}}{x-y}
f(2^ly)dy\\
&=T_{G_l,J_l}F(x).
\end{align*}
Hence 
\[
\|T_{g_l,I_l}f\|_{L^2}=\|T_{G_l,J_l}F(2^{-l} \cdot)\|_{L^2}=2^{\frac{l}{2}}\|T_{G_l,J_l}F\|_{L^2}\lesssim_K 2^{\frac{l}{2}}\|F\|_{L^2}=\|f\|_{L^2}.
\]
Here is a remark that the form of \(G_l\), \(\Gamma\) and \(\Phi\) comes from the change of variables above. It is natural to consider \(u=2^l x\) and \(v=2^l y\), then we are left to make the phase aligns with the form in Lemma \ref{ricci-stein}.

We have now obtained uniform bounds for all the truncated operators \(T_{g_l,I_l}\), and we need to put them together. Since there are infinitely many intervals, we cannot apply Lemma \ref{partition} directly, but we can use the dyadic structure to sum these pieces. We work with \(\bar T_{g,I}\), so that we can localize \(x\) and \(y\) separately.

Set \(\Omega=\{x:\gamma'(x)\in I\}\) and \(P_lf(x)=1_{\{\gamma'(x)\in I_l\}}f(x)\).
Since \(\bar T_{g,I}\) is supported on \(\Omega\times\Omega\), we have
\[
\bar T_{g,I}
=\sum_{l,m}P_l\bar T_{g,I}P_m
=T_{\mathrm{near}}+T_{\mathrm{far}}
\]
where
\[
T_{\mathrm{near}}
=\sum_{|l-m|\le1}P_l\bar T_{g,I}P_m,
\]
and
\[
T_{\mathrm{far}}
=\sum_{|l-m|\ge2}P_l\bar T_{g,I}P_m.
\]

First consider the far terms. If \(\gamma'(x)\in I_l\), \(\gamma'(y)\in I_m\), and \(|l-m|\ge2\), then the dyadic separation and the bounds \(1\le\gamma''\le2\) give
\[
|x-y|
\gtrsim|\gamma'(x)-\gamma'(y)|
\gtrsim \max \{\gamma'(x),\gamma'(y)\}\gtrsim \gamma'(x)+\gamma'(y).
\]
Taking absolute values and summing over the disjoint regions, we obtain
\[
|T_{\mathrm{far}}f(x)|
\lesssim1_{\Omega}(x)\int_{\Omega}
\frac{|f(y)|}{\gamma'(x)+\gamma'(y)}\,dy.
\]

Now change variables \(u=\gamma'(x)\) and \(v=\gamma'(y)\), and define
\[
F(v)=f\bigl((\gamma')^{-1}(v)\bigr).
\]

 Since both Jacobians are comparable to \(1\), Lemma \ref{hilbert} gives
\[
\begin{aligned}
\|T_{\mathrm{far}}f\|_{L_x^2}
&\lesssim
\left\|\int_0^\infty\frac{|F(v)|}{u+v}\,dv
\right\|_{L_u^2}\\
&\lesssim\|F\|_{L^2}
\lesssim\|f\|_{L^2}
\end{aligned}.
\]
For the near terms, let
\[
E_l=I_{l-1}\cup I_l\cup I_{l+1},
\;\;
Q_l=P_{l-1}+P_l+P_{l+1}.
\]
The same argument used for $I_l$ gives
\[
\|\bar T_{g,E_l}\|_{L^2\to L^2}\lesssim_K 1.
\]

Notice that
\[
T_{\mathrm{near}}f
=\sum_lP_l\bar T_{g,E_l}f.
\]
Moreover, the local estimate gives
\[
\|P_l\bar T_{g,E_l}f\|_2
\lesssim_K\|Q_lf\|_2.
\]
Since the output supports are disjoint,
\[
\begin{aligned}
\|T_{\mathrm{near}}f\|_2^2
&=\sum_l\|P_l\bar T_{g,E_l}f\|_2^2\\
&\lesssim_K\sum_l\|Q_lf\|_2^2
\lesssim\|f\|_2^2
\end{aligned},
\]
where the last inequality follows from the bounded overlap of the sets \(E_l\).
Combining the near and far estimates we have
\[
\|\bar T_{g,I}f\|_2\lesssim_K\|f\|_2.
\]
as desired.

\end{proof}

Now we should push forward a little more to deal with the case where \(g''=0\) on both endpoints. This follows by splitting the interval into two halves and applying the preceding lemma after translation or reflection.

\begin{theorem}
\label{two endpoints}
Let $I=(a,b)\subset [-2,2]$, and let $g\in C^3(I)$
, $g''$ does not vanish on $I$. Denote \(d_I(t):=\min(t-a,b-t)\). Assume that
\[
d_I(t)\left|\frac{g'''(t)}{g''(t)}\right|\le K.\]
Then
\[
\|T_{g,I}\|_{L^2\to L^2}\lesssim_K 1.
\]
\end{theorem}

\begin{proof}
Let $c=(a+b)/2$, and \(I_-=(a,c)\), \( I_+=(c,b)\).
By Lemma \ref{partition}, it suffices to prove the desired
estimate for $T_{g,I_-}$ and $T_{g,I_+}$.

On $I_-$, we have $d_I(t)=t-a$. Define \(u=t-a\), and \(g_-(u)=g(a+u)\) for \(0<u<c-a\), then $g_-''$ does not vanish, and
\[
u\left|\frac{g_-'''(u)}{g_-''(u)}\right|
=
(t-a)\left|\frac{g'''(t)}{g''(t)}\right|
\le K.
\]

Denote \(J_{-}=(0,c-a)\), Lemma \ref{one-endpoint} gives
\[
\|T_{g_-,J_-}\|_{L^2\to L^2}\lesssim_K 1.
\]

Let \(\Gamma(x)=\gamma(x)-ax\) and \(\Phi(x,y)=\frac{\Gamma(x)-\Gamma(y)}{x-y}\), then 
\begin{align*}
    T_{g,I_-}f(x)&=\int 1_{\{\varphi(x,y) \in I_-\}} 1_{\{|x-y|>\epsilon\}} \frac{e^{\mi (x-y)g(\varphi(x,y))}}{x-y}f(y) dy\\
    &=\int 1_{\{\Phi(x,y) \in J_-\}}   1_{\{|x-y|>\epsilon\}} \frac{e^{\mi (x-y)g_-(\Phi(x,y))}}{x-y}f(y) dy\\
    &= T_{g_-,J_-}f(x).
\end{align*}

Hence

\[
\|T_{g,I_-}\|_{L^2\to L^2}=\|T_{g_-,J_-}\|_{L^2\to L^2}\lesssim_K 1.
\]

The similar argument for $T_{g,I_+}$.

\end{proof}

Now we want to apply Theorem \ref{two endpoints} to our setting,
where $g$ is semialgebraic with complexity $k$.
First, we can divide its domain into $\lesssim_k 1$ intervals
on which $g\in C^3$ and $g''$ is either identically zero or
never zero. Moreover, by Theorem \ref{algebraic}, we can assume that \(g\) is algebraic. The case $g''\equiv0$ has already been treated.
Thus, by Lemma \ref{partition}, we only need to show that
each remaining interval can be divided into $\lesssim_k 1$
subintervals $J$ on which
\[
d_J(t)\left|\frac{g'''(t)}{g''(t)}\right|\le C_k.
\]

\bigskip
\section{A partition for algebraic functions}
\label{complex}
By the reduction above, we only need to prove the following:
if $g$ satisfies a polynomial equation of bounded degree and
$g''$ never vanishes, then its domain can be divided into
a bounded number of intervals $J$ on which
\[
d_J(t)\left|\frac{g'''(t)}{g''(t)}\right|
\]
is uniformly bounded. Both bounds should depend only on
the degree of the polynomial.

To prove this, we will use some facts about holomorphic functions.

\begin{lemma}[Bloch's Theorem]
\label{bloch}
Let $H$ be holomorphic on the unit disk $\mathbb D$, with
$H'(0)\ne0$. Then $H(\mathbb D)$ contains an open disk
of radius $|H'(0)|/72$.
\end{lemma}

\begin{lemma}
\label{finite-valency}
Let $h$ be holomorphic and nonvanishing on $B(t,r)$.
Suppose that, for every $w\in\mathbb C$, the equation $h(z)=w$
has at most $p$ distinct solutions in this disk. Then
\[
r\left|\frac{h'(t)}{h(t)}\right|\le72 \pi p.
\]
\end{lemma}

\begin{proof}
Since $B(t,r)$ is simply connected and $h$ never vanishes,
we may choose a holomorphic logarithm of $h$.
For $\zeta\in\mathbb D$, set
\[
H(\zeta)=\log h(t+r\zeta).
\]
Then
\[
|H'(0)|=r\left|\frac{h'(t)}{h(t)}\right|.
\]
If this quantity were greater than $24\pi p$,
Lemma \ref{bloch} would imply that $H(\mathbb D)$ contains
a disk of radius greater than $\pi p$.
If $c$ is its center, this disk contains the $p+1$ points
\[
c+2\pi i(j-p/2),\;\;
 j=0,\dots,p.
\]
These points have the same exponential.
Choosing a preimage under $H$ for each point therefore gives
$p+1$ distinct points at which $h$ takes the same value.
This contradicts the assumption.

\end{proof}

We next apply this estimate to an algebraic function.

\begin{lemma}
\label{algebraic-log-derivative}
Let $h$ be a nonconstant real analytic function on a bounded
open interval $I$. Suppose that
\[
Q(t,h(t))=0
\]
where $Q\in\mathbb C[z,w]$ is irreducible over $\mathbb C$
and has total degree at most $D$.
Then $I$ can be divided, up to finitely many points, into
$O(D^2)$ open intervals $J$ on which $h$ never vanishes and,
for every $t\in J$,
\[
d_J(t)\left|\frac{h'(t)}{h(t)}\right|\le72\pi D.
\]
\end{lemma}

\begin{proof}
Write
\[
Q(z,w)=\sum_{j=0}^n a_j(z)w^j.
\]
Since $h$ is nonconstant and $Q$ is irreducible,
$n\ge1$, $a_n\not\equiv0$, and $a_0\not\equiv0$.
Also, the polynomial
\[
\Delta(z):=\operatorname{Res}_w(Q,\partial_wQ)
\]
is not identically zero.

Let $E$ be the set of complex zeros of $a_na_0\Delta$.
The degree bound for the resultant gives
\[
\#E\le2D+(2n-1)D\lesssim D^2.
\]
Divide $I$ at the real parts of the points of $E$ that lie in $I$.
This produces $O(D^2)$ intervals.

Fix one such interval $J$, a point $t\in J$, and
$0<r<d_J(t)$. For every $e\in E$,
\[
|t-e|\ge|t-\operatorname{Re}e|\ge d_J(t)>r.
\]
Thus $a_n$, $a_0$, and $\Delta$ do not vanish on $B(t,r)$.

Since $a_n$ and $\Delta$ are nonzero, the polynomial
$Q(z,\cdot)$ has exactly $n$ distinct roots at every point
of the disk. These roots define local holomorphic functions. Since \(h\) is continuous and the roots are distinct, \(h\) agrees on a real neighborhood of \(t\) with one of these local holomorphic root functions.
They extend to single-valued holomorphic branches on the
whole disk by the monodromy theorem, since the disk is
simply connected.
This branch never vanishes, since $a_0$ is nonzero there.

For each fixed $w_0\in\mathbb C$, the polynomial $Q(z,w_0)$
is not identically zero. Otherwise $w-w_0$ would divide
$Q(z,w)$, and irreducibility would force $h$ to be constant.
Every solution of $h(z)=w_0$ is therefore a zero of a nonzero
polynomial in $z$ of degree at most $D$.
Hence each value has at most $D$ distinct preimages.

Lemma \ref{finite-valency} now gives
\[
r\left|\frac{h'(t)}{h(t)}\right|\le72\pi D.
\]
Letting $r$ increase to $d_J(t)$ proves the claim.
\end{proof}

We now apply this result to $h=g''$.
For this, we need a polynomial equation for $g''$
whose degree is controlled by that of $g$.

\begin{lemma}
\label{algebraic-second-derivative}
Let $g$ be real analytic on an interval $I$, and suppose that
\[
P(t,g(t))=0,
\]
where $P\in\mathbb C[z,w]$ is irreducible and has total
degree at most $d$.
If $g''$ is nonconstant, then there is an irreducible polynomial
$Q\in\mathbb C[z,v]$ such that
\[
Q(t,g''(t))=0
\]
and
\[
\deg Q\le6d^2.
\]
\end{lemma}

\begin{proof}
At points where $P_w(t,g(t))\ne0$, implicit differentiation gives
\[
g''(t)
=-\frac{N(t,g(t))}{P_w(t,g(t))^3}
\]
where
\[
N=P_{zz}P_w^2-2P_{zw}P_zP_w+P_{ww}P_z^2.
\]
Consequently, $h=g''$ satisfies
\[
R(t,h(t))=0
\]
where
\[
R(z,v)=\operatorname{Res}_w
\bigl(P(z,w),vP_w(z,w)^3+N(z,w)\bigr).
\]

This resultant is not identically zero.
Indeed, over an algebraic closure of $\mathbb C(z)$,
the roots $\alpha_j$ of $P(z,\cdot)$ are distinct.
Up to a nonzero factor, the resultant is
\[
\prod_j
\bigl(vP_w(z,\alpha_j)^3+N(z,\alpha_j)\bigr).
\]
Each factor has a nonzero coefficient of $v$, so the product
is a nonzero polynomial in $v$.

The relation $R(t,h(t))=0$ holds away from the finitely many
zeros of $\operatorname{Res}_w(P,P_w)$, and extends to all
of $I$ by continuity.
The standard degree bound for the resultant gives
\[
\deg R
\le (3d-3)d+d(3d-2)
\le6d^2.
\]
Choose an irreducible factor $Q$ of $R$ which vanishes on
the analytic germ of $h$ at some point.
The identity theorem then gives $Q(t,h(t))=0$ throughout $I$.
\end{proof}

\begin{theorem}
\label{curvature-partition}
Let $g\in C^3(I)$ be real-valued on a bounded open interval $I$.
Suppose that $g''$ never vanishes and that
\[
P(t,g(t))=0
\]
for a nonzero polynomial $P$ of total degree at most $d$.
Then $I$ can be divided, up to finitely many points, into
$\lesssim_d1$ intervals $J$ on which
\[
d_J(t)\left|\frac{g'''(t)}{g''(t)}\right|\le C_d.
\]
\end{theorem}

\begin{proof}
After a further partition of \(I\) into \(\lesssim_d 1\) intervals, we may assume that \(P\) is irreducible, with degree still at most \(d\), and \(g\) is analytic.

If $g''$ is constant, then $g'''=0$, and the conclusion is immediate.
Otherwise, Lemma \ref{algebraic-second-derivative} gives an
irreducible polynomial $Q$ such that
\[
Q(t,g''(t))=0
\]
with $\deg Q\lesssim_d 1$.
Applying Lemma \ref{algebraic-log-derivative} to $h=g''$,
we obtain a partition into $\lesssim_d 1$ intervals $J$ on which
\[
d_J(t)\left|\frac{g'''(t)}{g''(t)}\right|\le C_d.
\]
This proves the claim.

\end{proof}

Then we obtain the whole proof by Theorem \ref{curvature-partition}.

\printbibliography

\end{document}